\documentclass[reqno,11pt]{amsart}
\usepackage{amsmath,amsfonts,amscd,amssymb,latexsym}
\usepackage{hyperref}
\usepackage[nobysame]{amsrefs}
\usepackage{comment}
\usepackage{soul}
\usepackage{tikz}
\usepackage{anysize}
\usepackage{bbm}
\usepackage{cancel}
\usepackage[normalem]{ulem}

\numberwithin{equation}{section}

\theoremstyle{plain}
\newtheorem{theorem}{Theorem}[section]
\newtheorem{lemma}{Lemma}[section]
\newtheorem{corollary}{Corollary}[section]
\newtheorem{proposition}{Proposition}[section]

\theoremstyle{definition}
\newtheorem{definition}{Definition}[section]

\newtheorem{assumption}{Assumption}[section]

\theoremstyle{remark}
\newtheorem{remark}{Remark}[section]

\newcommand{\ind}{\operatorname{ind}}
\newcommand{\id}{\operatorname{id}}
\newcommand{\dom}{\operatorname{dom}}
\newcommand{\loc}{\operatorname{loc}}
\newcommand{\cc}{{cc}}
\newcommand{\dm}{\partial M}
\newcommand{\RR}{\mathbb{R}}
\newcommand{\ZZ}{\mathbb{Z}}

\newcommand{\AAA}{\mathcal{A}}

\newcommand{\DD}{\mathcal{D}}
\newcommand{\hDD}{\hat{\mathcal{D}}}

\newcommand{\upper}{\uppercase\expandafter}

\begin{document}

\title[Callias operators with APS boundary conditions]{The index of Callias-type operators with Atiyah-Patodi-Singer boundary conditions}

\author{Pengshuai Shi}
\address{Department of Mathematics,
Northeastern University,
Boston, MA 02115,
USA}

\email{shi.pe@husky.neu.edu}


\begin{abstract}
We compute the index of a Callias-type operator with APS boundary condition on a manifold with compact boundary in terms of combination of indexes of induced operators on a compact hypersurface. Our result generalizes the classical Callias-type index theorem to manifolds with compact boundary. 
\end{abstract}

\maketitle

\section{Introduction}\label{S:intro}

Constantine Callias in \cite{Callias78}, considered a class of perturbed Dirac operators on an odd-dimensional Euclidean space which are Fredholm and found a beautiful formula for the  index of such operators. This result was soon generalized to Riemannian manifolds by many authors,  \cite{BottSeeley78}, \cite{BruningMoscovici},  \cite{Anghel93}, \cite{Raade}, \cite{Bunke}. A nice character of the Callias index theorem is that it reduces a noncompact index to a compact one. Recently, many new properties, generalizations and applications of Callias-type index were found, cf., for example, \cite{Kottke11}, \cite{CarvalhoNistor14}, \cite{Wimmer14}, \cite{Kottke15}, \cite{BrCecchini16}, \cite{BrShi}. 

In this paper we extend the Callias-type index theory to manifolds with compact boundary. The study of the index theory on compact manifolds with boundary was initiated in \cite{AB}. In the seminal paper \cite{APS}, Atiyah, Patodi and Singer computed the index of a first order elliptic operator with a non-local boundary condition. This so-called Atiyah-Patodi-Singer (APS) boundary condition is defined using the spectrum of a self-adjoint operator associated to the restriction of the original operator to the boundary. The Atiyah-Patodi-Singer index theorem inspired an intensive study of boundary value problems for first-order elliptic operators, especially Dirac-type operators (see \cite{BW} for compact manifolds). Recently, B\"ar and Ballmann in \cite{BB-1} gave a thorough description of boundary value problems for first-order elliptic operators on (not necessarily compact) manifolds with compact boundary. They obtained the Fredholm property for Callias-type operators with APS boundary conditions, making it possible to study the index problem on noncompact manifolds with boundary. The results in \cite{BB-1} were also partially generalized to Spin$^c$ manifolds of bounded geometry with noncompact boundary in \cite{GN}.

In this paper we combine the results  of \cite{APS}, \cite{BB-1} and \cite{Callias78} and  compute the index of Callias-type operators with APS boundary conditions. We show that this index is equal to a combination of indexes of the induced operators on a compact hypersurface and a boundary term which appears in APS index theorem. Thus our result generalizes the Callias index theorem to manifolds with boundary. We point out that our proof technique leads to a new proof of the classical (boundaryless) Callias index theorem. Recently the results of this paper were partially extended to the case of noncompact boundary in \cite{BrShi17, BrShi17-2}.

The paper is organized as follows. In Section \ref{S:mfldwbd}, we introduce the basic setting for manifolds with compact boundary. In Section \ref{S:preliminary}, we discuss some results from \cite{BB-1} about boundary value problems of Dirac-type operators with the focus on APS boundary condition. Also, we recall the splitting theorem and relative index theorem which will play their roles in proving the main theorem. Then in Section \ref{S:APS-Callias}, we study the above-mentioned APS-Callias index problem and give our main result in Theorem \ref{T:Calliasindthm}, followed by some consequences. The theorem is proved in Section \ref{S:proof}.

\subsection*{Acknowledgment} I am very grateful to Prof. Maxim Braverman for bringing this problem to my attention and offering valuable suggestions. I would also like to thank Simone Cecchini and Chris Kottke for helpful discussions. Finally, I appreciate the constructive comments and suggestions by the reviewers.

\section{Manifolds with compact boundary}\label{S:mfldwbd}

We introduce the basic notations that will be used later.

\subsection{Setting}\label{SS:setting}

Let $M$ be a Riemannian manifold with compact boundary $\dm$. We assume the manifold is complete in the sense of metric spaces and call it a \emph{complete} Riemannian manifold throughout this paper. We denote by $dV$ the volume element on $M$ and by $dS$ the volume element on $\dm$. The interior of $M$ is denoted by $\mathring{M}$. For a vector bundle $E$ over $M$, $C^\infty(M,E)$ is the space of smooth sections of $E$, $C_c^\infty(M,E)$ is the space of smooth sections of $E$ with compact support, and $C_\cc^\infty(M,E)$ is the space of smooth sections of $E$ with compact support in $\mathring{M}$. Note that
\[
C_\cc^\infty(M,E)\subset C_c^\infty(M,E)\subset C^\infty(M,E).
\]
When $M$ is compact, $C_c^\infty(M,E)=C^\infty(M,E)$; when $\dm=\emptyset$, $C_\cc^\infty(M,E)=C_c^\infty(M,E)$. We denote by $L^2(M,E)$ the Hilbert space of square-integrable sections of $E$, which is the completion of $C_c^\infty(M,E)$ with respect to the norm induced by the $L^2$-inner product
\[
(u_1,u_2)\;:=\;\int_M\langle u_1,u_2\rangle dV,
\]
where $\langle\cdot,\cdot\rangle$ denotes the fiberwise inner product.

Let $E,F$ be two Hermitian vector bundles over $M$ and $D:C^\infty(M,E)\to C^\infty(M,F)$ be a first-order differential operator. The \emph{formal adjoint} of $D$, denoted by $D^*$, is defined by
\[
\int_M\langle Du,v\rangle dV\;=\;\int_M\langle u,D^*v\rangle dV,
\]
for all $u\in C_\cc^\infty(M,E)$ and $v\in C^\infty(M,F)$. If $E=F$ and $D=D^*$, then $D$ is called \emph{formally self-adjoint}.

\subsection{Minimal and maximal extensions}\label{SS:minmaxext}

Suppose $D_\cc:=D|_{C_\cc^\infty(M,E)}$, and view it as an unbounded operator from $L^2(M,E)$ to $L^2(M,F)$. The \emph{minimal extension} $D_{\min}$ of $D$ is the operator whose graph is the closure of that of $D_\cc$. The \emph{maximal extension} $D_{\max}$ of $D$ is defined to be $D_{\max}=\big((D^*)_{cc}\big)^{\rm ad}$, where ``ad'' means adjoint of the operator in the sense of functional analysis. Both $D_{\min}$ and $D_{\max}$ are closed operators. Their domains, $\dom D_{\min}$ and $\dom D_{\max}$, become Hilbert spaces equipped with the \emph{graph norm}, which is the norm associated with the inner product
\[
(u_1,u_2)_D\;:=\;\int_M(\langle u_1,u_2\rangle+\langle Du_1,Du_2\rangle)dV.
\]

\subsection{Green's formula}\label{SS:greensfor}

Let $\tau\in TM|_{\dm}$ be the unit inward normal vector field along $\dm$. Using the Riemannian metric, $\tau$ can be identified with its associated one-form. We have the following formula (cf. \cite[Proposition 3.4]{BW}).

\begin{proposition}[Green's formula] \label{P:greensfor}
Let $D$ be as above. Then for all $u\in C_c^\infty(M,E)$ and $v\in C_c^\infty(M,F)$,
\begin{equation}\label{E:greensfor}
\int_M\langle Du,v\rangle dV\;=\;\int_M\langle u,D^*v\rangle dV\,-\,\int_{\dm}\langle\sigma_D(\tau)u,v\rangle dS,
\end{equation}
where $\sigma_D$ denotes the principal symbol of the operator $D$.
\end{proposition}

\begin{remark}\label{R:greensfor}
By \cite[Theorem 6.7]{BB-1}, the formula \eqref{E:greensfor} can be generalized to the case where $u\in\dom D_{\max}$ and $v\in\dom(D^*)_{\max}$.
\end{remark}

\subsection{Sobolev spaces}\label{SS:sobsp}

Let $\nabla^E$ be a Hermitian connection on $E$. For any $u\in C^\infty(M,E)$, the covariant derivative $\nabla^E u\in C^\infty(M,T^*M\otimes E)$. For $k\in\ZZ_+$, we define the \emph{$k^{th}$ Sobolev space}
\[
H^k(M,E)\;:=\;\{u\in L^2(M,E):\nabla^E u,(\nabla^E)^2u,\dots,(\nabla^E)^ku\in L^2(M)\},
\]
where the covariant derivatives are understood in distributional sense. It is a Hilbert space with $H^k$-norm
\[
\|u\|_{H^k(M)}^2\;:=\;\|u\|_{L^2(M)}^2+\|\nabla^E u\|_{L^2(M)}^2+\cdots+\|(\nabla^E)^ku\|_{L^2(M)}^2.
\]
Note that when $M$ is compact, $H^k(M,E)$ does not depend on the choices of $\nabla^E$ and Riemannian metric, but when $M$ is noncompact, it does.

We say $u\in L_{\loc}^2(M,E)$ if the restrictions of $u$ to compact subsets of $M$ have finite $L^2$-norm. For $k\in\ZZ_+$, we say $u\in H_{\loc}^k(M,E)$, the \emph{$k^{th}$ local Sobolev space}, if $u,\nabla^Eu,(\nabla^E)^2u,\dots,(\nabla^E)^ku$ all lie in $L_{\rm loc}^2(M,E)$. This Sobolev space is independent of the preceding choices.
 
Similarly, we fix a Hermitian connection on $F$ and define the spaces $L^2(M,F)$, $L^2_{\loc}(M,F)$, $H^k(M,F)$, and $H^k_{\loc}(M,F)$.

\section{Preliminary results}\label{S:preliminary}

In this section, we summarize some results on boundary value problems on complete manifolds with compact boundary. We mostly follow \cite{BB-1, BB-2}.

\subsection{Adapted operators to Dirac-type operators}\label{SS:adapop}

Let $E$ be a Clifford module over $M$ with Clifford multiplication denoted by $c(\cdot)$. We say that $D:C^\infty(M,E)\to C^\infty(M,E)$ is a \emph{Dirac-type operator} if the principal symbol of $D$ is $c(\cdot)$. In local coordinates, $D$ can be written as
\begin{equation}\label{E:Diracopr}
D\;=\;\sum_{j=1}^nc(e_j)\nabla_{e_j}^E+V
\end{equation}
at $x\in M$, where $e_1,\dots,e_n$ is an orthonormal basis of $T_xM$ (using Riemannian metric to identify $TM$ and $T^*M$), $\nabla^E$ is a Hermitian connection on $E$ and $V\in{\rm End}(E)$ is the \emph{potential}. When $V=0$, $D$ is called \emph{Dirac operator}.

The formal adjoint $D^*$ of a Dirac-type operator $D$ is also of Dirac type. Note that for $x\in\dm$, one can identify $T_x^*\dm$ with the space $\{\xi\in T_x^*M:\langle\xi,\tau(x)\rangle=0\}$.

\begin{definition}\label{D:adaptedopr}
A formally self-adjoint first-order differential operator $A:C^\infty(\dm,E)\to C^\infty(\dm,E)$ is called an \emph{adapted operator} to $D$ if the principal symbol of $A$ is given by
\[
\sigma_A(\xi)\;=\;\sigma_D(\tau(x))^{-1}\circ\sigma_D(\xi).
\]
\end{definition}

\begin{remark}\label{R:adaptedopr}
Adapted operators always exist and are also of Dirac type. They are unique up to addition of a Hermitian bundle map of $E$ (cf. \cite[Section 3]{BB-2}).
\end{remark}

If $A$ is adapted to $D$, then
\begin{equation}\label{E:tildeA}
\tilde A\;=\;c(\tau)\circ(-A)\circ c(\tau)^{-1}
\end{equation}
is an adapted operator to $D^*$. Moreover, if $D$ is formally self-adjoint, we can find an adapted operator $A$ to $D$ such that
\begin{equation}\label{E:anticommuting}
A\circ c(\tau)\;=\;-c(\tau)\circ A,
\end{equation}
and, hence, $\tilde{A}= A$.

By definition, $A$ is an essentially self-adjoint elliptic operator on the closed manifold $\dm$. Hence $A$ has discrete spectrum consisting of real eigenvalues $\{\lambda_j\}_{j\in\ZZ}$, each of which has finite multiplicity. In particular, the corresponding eigenspaces $V_j$ are finite-dimensional. Thus we have decomposition of $L^2(\dm,E)$ into a direct sum of eigenspaces of $A$:
\begin{equation}\label{E:L2decom}
L^2(\dm,E)\;=\;\bigoplus\limits_{\lambda_j\in{\rm spec}(A)}V_j.
\end{equation}
For any $s\in\RR$, the positive operator $(\id+A^2)^{s/2}$ is defined by functional calculus. Then the \emph{$H^s$-norm} on $C^\infty(\dm,E)$ is defined by
\[
\|u\|_{H^s(\dm,E)}^2\;:=\;\|(\id+A^2)^{s/2}u\|_{L^2(\dm,E)}^2.
\]
The Sobolev space $H^s(\dm,E)$ is the completion of $C^\infty(\dm,E)$ with respect to this norm.

\begin{remark}\label{R:sobequiv}
When $s\in\ZZ_+$, this definition of Sobolev spaces coincides with that of Subsection \ref{SS:sobsp} via covariant derivatives.
\end{remark}

For $I\subset\RR$, let
\begin{equation}\label{E:specproj}
P_I\;:\;L^2(\dm,E)\;\to\;\bigoplus\limits_{\lambda_j\in I}V_j
\end{equation}
be the orthogonal spectral projection. It's easy to see that
\[
P_I(H^s(\dm,E))\;\subset\;H^s(\dm,E)
\]
for all $s\in\RR$. Set $H_I^s(A):=P_I(H^s(\dm,E))$. For $a\in\RR$, we define the \emph{hybrid} Sobolev space
\begin{equation}\label{E:checkH}
\check{H}(A)\;:=\;H_{(-\infty,a)}^{1/2}(A)\;\oplus\;H_{[a,\infty)}^{-1/2}(A)
\end{equation}
with $\check{H}$-norm
\[
\|u\|_{\check{H}(A)}^2\;:=\;
\|P_{(-\infty,a)}u\|_{H^{1/2}(\dm,E)}^2\;+\;
\|P_{[a,\infty)}u\|_{H^{-1/2}(\dm,E)}^2.
\]
The space $\check{H}(A)$ is independent of the choice of $a$ (cf. \cite[p. 27]{BB-1}).

\subsection{Boundary value problems}\label{SS:bvp}

Let $D$ be a Dirac-type operator. If $\dm=\emptyset$, then $D$ has a unique extension, i.e., $D_{\min}=D_{\max}$. (When $D$ is formally self-adjoint, this is called \emph{essentially self-adjointness}, cf. \cite{Chernoff}, \cite[Theorem 1.17]{GL}.) But when $\dm\ne\emptyset$, the minimal and maximal extensions may not be equal. Those closed extensions lying between $D_{\min}$ and $D_{\max}$ give rise to boundary value problems.

One of the main results of \cite{BB-1} is the following.

\begin{theorem}\label{T:bc}
For any closed subspace $B\subset\check{H}(A)$, denote by $D_B$ the extension of $D$ with domain
\[
\dom D_B\;=\;\{u\in\dom D_{\max}:u|_{\dm}\in B\}.
\]
Then $D_B$ is a closed extension of $D$ between $D_{\min}$ and $D_{\max}$, and any closed extension of $D$ between $D_{\min}$ and $D_{\max}$ is of this form.
\end{theorem}

\begin{remark}
We recall the \emph{trace theorem} which says that the trace map $\cdot|_{\dm}:C_c^\infty(M,E)\to C^\infty(\dm,E)$ extends to a bounded linear map
\[
T\;:\;H_{\loc}^k(M,E)\;\to\;H^{k-1/2}(\dm,E)
\]
for all $k\ge1$.
\end{remark}

Due to this theorem, one can define boundary conditions in the following way.

\begin{definition}\label{D:bc}
A \emph{boundary condition} for $D$ is a closed subspace of $\check{H}(A)$. We use the notation $D_B$ from Theorem \ref{T:bc} to denote the operator $D$ with boundary condition $B$.
\end{definition}

Regarding $D_B$ as an unbounded operator on $L^2(M,E)$, its adjoint operator is $D_{B^{\rm ad}}^*$, where the boundary condition is
\[
B^{\rm ad}\;=\;\{v\in\check{H}(\tilde A):(\sigma_D(\tau)u,v)=0,\mbox{ for all }u\in B\},
\]
and $\tilde A$ is an adapted operator to $D^*$.

\subsection{Elliptic boundary conditions}\label{SS:ellbc}

Notice that for general boundary conditions, $\dom D_B\not\subset H_{\loc}^1(M,E)$.

\begin{definition}\label{D:ellbc}
A boundary condition $B$ is said to be \emph{elliptic} if $\dom D_B\subset H_{\loc}^1(M,E)$ and $\dom D_{B^{\rm ad}}^*\subset H_{\loc}^1(M,E)$.
\end{definition}

\begin{remark}\label{R:ellbc}
This definition is equivalent
to saying that $B\subset H^{1/2}(\dm,E)$ and its adjoint boundary condition $B^{\rm ad}\subset H^{1/2}(\dm,E)$ {(cf. \cite[Theorem 1.7]{BB-1})}. There is another equivalent but technical way to define elliptic boundary conditions, see \cite[Definition 7.5]{BB-1} or \cite[Definition 4.7]{BB-2}. From \cite{BB-1,BB-2}, $B$ is an elliptic boundary condition if and only if $B^{\rm ad}$ is.
\end{remark}

The definition of elliptic boundary condition can be generalized as follows.

\begin{definition}\label{D:higherellbc}
A boundary condition $B$ is said to be
\begin{enumerate}
\item \emph{$m$-regular}, where $m\in\ZZ_+$, if
\[
\begin{aligned}
D_{\max}u\in H_{\loc}^k(M,E)\;&\Longrightarrow\;u\in H_{\loc}^{k+1}(M,E), \\
D_{\max}^*v\in H_{\loc}^k(M,E)\;&\Longrightarrow\;v\in H_{\loc}^{k+1}(M,E)
\end{aligned}
\]
for all $u\in\dom D_B$, $v\in\dom D_{B^{\rm ad}}^*$, and $k=0,1,\dots,m-1$.
\item \emph{$\infty$-regular} if it is $m$-regular for all $m\in\ZZ_+$.
\end{enumerate}
\end{definition}

\begin{remark}\label{R:higherellbc}
By this definition, an elliptic boundary condition is 1-regular.
\end{remark}

It is clear that if $B$ is an $\infty$-regular boundary condition, then 
\[
\ker D_B\subset C^\infty(M,E),\quad\ker D_{B^{\rm ad}}^*\subset C^\infty(M,E).
\]

\subsection{The Atiyah-Patodi-Singer boundary condition}\label{SS:APSbc}

A typical example of elliptic boundary condition, which is called Atiyah-Patodi-Singer boundary condition (or APS boundary condition), is introduced in \cite{APS}.

Let $D:C^\infty(M,E)\to C^\infty(M,E)$ be a Dirac-type operator. Assume the Riemannian metric and the Clifford module $E$ (with the associated Clifford multiplication and Clifford connection) have product structure near the boundary $\dm$. So $D$ can be written as
\begin{equation}\label{E:prostrofD}
D\;=\;c(\tau)\big(\partial_t+A+R\big)
\end{equation}
in a tubular neighborhood of $\dm$, where $t$ is the normal coordinate, $A$ is an adapted operator to $D$, and $R$ is a zeroth order operator on $\dm$. Then
\[
D^*\;=\;c(\tau)\big(\partial_t+\tilde A+\tilde R\big),
\]
where $\tilde A$ is as in \eqref{E:tildeA}. When $D=D^*$, one can choose $R=\tilde R=0$ so that $A=\tilde A$.

Let $P_{(-\infty,0)}$ be the spectral projection as in \eqref{E:specproj} and set
\[
	H_{(-\infty,0)}^{1/2}(A)\;=\;P_{(-\infty,0)}(H^{1/2}(\dm,E)).
\]

\begin{definition}\label{D:APSbc}
The \emph{Atiyah-Patodi-Singer boundary condition} is
\begin{equation}\label{E:APSbc}
B_{\rm APS}\;:=\;H_{(-\infty,0)}^{1/2}(A).
\end{equation}
This is a closed subspace of $\check{H}(A)$ (recall that the space $\check{H}(A)$ is defined in \eqref{E:checkH}). The adjoint boundary condition is given by
\begin{equation}\label{E:adAPS}
B_{\rm APS}^{\rm ad}\;=\;c(\tau)H_{[0,\infty)}^{1/2}(A)\;=\;H_{(-\infty,0]}^{1/2}(\tilde A).
\end{equation}
\end{definition}

By \cite[Proposition 7.24 and Example 7.27]{BB-1}, we have that

\begin{proposition}\label{P:APS}
The APS boundary condition \eqref{E:APSbc} is an $\infty$-regular boundary condition.
\end{proposition}

\subsection{Invertibility at infinity}\label{SS:invatinf}

If the manifold $M$ is noncompact without boundary, in general, an elliptic operator on it is not Fredholm. Similarly, for noncompact manifold $M$ with compact boundary, an elliptic boundary condition does not guarantee that the operator is Fredholm. We now define a class of operators on noncompact manifolds which are Fredholm.

\begin{definition}\label{D:invatinf}
We say that an operator $D$ is \emph{invertible at infinity} (or \emph{coercive at infinity}) if there exist a constant $C>0$ and a compact subset $K\Subset M$ such that
\begin{equation}\label{E:invatinf}
\|Du\|_{L^2(M)}\;\ge\;C\|u\|_{L^2(M)},
\end{equation}
for all $u\in C_c^\infty(M,E)$ with ${\rm supp}(u)\cap K=\emptyset$.
\end{definition}

\begin{remark}
(\romannumeral1) By definition, if $M$ is compact, then $D$ is invertible at infinity.

(\romannumeral2) Boundary conditions have nothing to do with invertibility at infinity since the compact set $K$ can always be chosen such that a neighborhood of $\dm$ is contained in $K$.
\end{remark}

An important class of examples for operators which are invertible at infinity is the so-called \emph{Callias-type operators} that will be discussed in next section.

\subsection{Fredholmness}\label{SS:fred}

Recall that for $\dm=\emptyset$, a first-order essentially self-adjoint elliptic operator which is invertible at infinity is Fredholm (cf. \cite[Theorem 2.1]{Anghel93}). For $\dm\ne\emptyset$, we have the following analogous result (\cite[Theorem 8.5, Corollary 8.6]{BB-1}).

\begin{proposition}\label{P:fred}
Assume that $D_B:\dom D_B\to L^2(M,E)$ is a Dirac-type operator with elliptic boundary condition.
\begin{enumerate}
\item\label{I:1} If $D$ is invertible at infinity, then $D_B$ has finite-dimensional kernel and closed range.
\item\label{I:2} If $D$ and $D^*$ are invertible at infinity, then $D_B$ is a Fredholm operator.
\end{enumerate}
\end{proposition}

\begin{remark}
Since for an elliptic boundary condition $B$, $\dom D_B\subset H_{\loc}^1(M,E)$, the proof is essentially the same as that for the case without boundary (involving Rellich embedding theorem). And it is easy to see that \ref{I:2} is an immediate consequence of \ref{I:1}.
\end{remark}

Under the hypothesis of Proposition \ref{P:fred}.\ref{I:2}, we define the \emph{index} of $D$ subject to the boundary condition $B$ as the integer
\[
\ind D_B\;:=\;\dim\ker D_B-\dim\ker D_{B^{\rm ad}}^*\;\in\;\ZZ.
\]

\subsection{The splitting theorem}\label{SS:splittingthm}

We recall the splitting theorem of \cite{BB-1} which can be thought of as a more general version of \cite[Proposition 2.3]{Raade}. Let $D:C^\infty(M,E)\to C^\infty(M,E)$ be a Dirac-type operator on $M$. Let $N$ be a closed and two-sided hypersurface in $M$ which does not intersect the compact boundary $\dm$. Cut $M$ along $N$ to obtain a manifold $M'$, whose boundary $\dm'$ consists of disjoint union of $\dm$ and two copies $N_1$ and $N_2$ of $N$. One can pull back $E$ and $D$ from $M$ to $M'$ to define the bundle $E'$ and operator $D'$. Then $D':C^\infty(M',E')\to C^\infty(M',E')$ is still a Dirac-type operator. Assume that there is a unit inward normal vector field $\tau$ along $N_1$ and choose an adapted operator $A$ to $D'$ along $N_1$. Then $-A$ is an adapted operator to $D'$ along $N_2$.

\begin{theorem}[\cite{BB-1}, Theorem 8.17]\label{T:splittingthm}
Let $M,D,M',D'$ be as above. 

\begin{enumerate}
\item
$D$ and $D^*$ are invertible at infinity if and only if $D'$ and $(D')^*$ are invertible at infinity. 

\item
Let $B$ be an elliptic boundary condition on $\dm$. Fix $a\in\RR$ and let $B_1= H_{(-\infty,a)}^{1/2}(A)$ and $B_2= H_{[a,\infty)}^{1/2}(A)$ be boundary conditions along $N_1$ and $N_2$, respectively. Then the operators $D_B$ and $D'_{B\oplus B_1\oplus B_2}$ are Fredholm operators and
\[
	\ind D_B\;=\;\ind D'_{B\oplus B_1\oplus B_2}.
\]
\end{enumerate}
\end{theorem}

\subsection{Relative index theorem}\label{SS:rit}

Let $M_j$, $j=1,2$ be two {complete} manifolds
with compact boundary and $D_{j,B_j}:\dom D_{j,B_j}\to L^2(M_j,E_j)$ be two Dirac-type operators with elliptic boundary conditions. Suppose $M_j'\cup_{N_j}M_j''$ are partitions of $M_j$ into relatively open submanifolds, where $N_j$ are closed hypersurfaces of $M_j$ that do not intersect the boundaries. We assume that $N_j$ have tubular neighborhoods which are diffeomorphic to each other and the structures of $E_j$ (resp. $D_j$) on the neighborhoods are isomorphic.

Cut  $M_j$ along $N_j$ and glue the pieces together interchanging $M_1''$ and $M_2''$. In this way we obtain the manifolds 
\[
M_3\ :=\ M_1'\cup_N M_2'',\qquad\qquad M_4\ :=\ M_2'\cup_N M_1'',
\]
where $N\cong N_1\cong N_2$. Then we get operators $D_{3,B_3}$ and $D_{4,B_4}$ on $M_3$ and $M_4$, respectively. The following relative index theorem, which generalizes \cite[Theorem 8.19]{BB-1}, is a direct consequence of Theorem \ref{T:splittingthm}. (One can see \cite[Theorem 1.2]{Bunke} for a boundaryless version.)

\begin{theorem}\label{T:rit}
If $D_j$ and $D_j^*$, $j=1,2,3,4$ are all invertible at infinity, then $D_{j,B_j}$ are all Fredholm operators, and
\[
\ind D_{1,B_1}+\ind D_{2,B_2}\;=\;\ind D_{3,B_3}+\ind D_{4,B_4}.
\]
\end{theorem}

\begin{proof}
Clearly the hypersurfaces $N_j$ satisfy the hypothesis of Theorem \ref{T:splittingthm}. As in last subsection, choose boundary conditions $B_{N_j}'$ and $B_{N_j}''$ along $N_j$ on $M_j'$ and $M_j''$, respectively. Since $D_j$ and $D_j^*$ are invertible at infinity, from Theorem \ref{T:splittingthm},
\[
\ind D_{j,B_j}\;=\;\ind D_{j,B_j'\oplus B_{N_j}'}'+\ind D_{j,B_j''\oplus B_{N_j}''}'',\qquad j=1,2,
\]
where $B_j'$ and $B_j''$ are the restrictions of the boundary condition $B_j$ to $M_j'$ and $M_j''$, respectively. By the construction of $M_3$ and $M_4$,
\[
\ind D_{3,B_3}\;=\;\ind D_{1,B_1'\oplus B_{N_1}'}'+\ind D_{2,B_2''\oplus B_{N_2}''}'',
\]
\[
\ind D_{4,B_4}\;=\;\ind D_{2,B_2'\oplus B_{N_2}'}'+\ind D_{1,B_1''\oplus B_{N_1}''}''.
\]
Adding together, the theorem is proved.
\end{proof}

\section{Callias-type operators with APS boundary conditions}\label{S:APS-Callias}

\subsection{Callias-type operators}\label{SS:Calliasopr}

Let $M$ be a complete \emph{odd-dimensional} Riemannian manifold with boundary $\dm$. Suppose that $E$ is a Clifford module over $M$. Let $D:C^\infty(M,E)\to C^\infty(M,E)$ be a formally self-adjoint Dirac-type operator. Suppose $\Phi\in{\rm End}(E)$ is a self-adjoint bundle map (called \emph{Callias potential}). Then $\DD:=D+i\Phi$ is again a Dirac-type operator on $E$ with formal adjoint given by
\[
\DD^*\;=\;D-i\Phi.
\]
So
\begin{equation}\label{E:Calliassq}
\begin{aligned}
\DD^*\DD\;&=\;D^2+\Phi^2+i[D,\Phi], \\
\DD\DD^*\;&=\;D^2+\Phi^2-i[D,\Phi],
\end{aligned}
\end{equation}
where
\[
	[D,\Phi]\;:=\;D\Phi\;-\;\Phi D
\]
is the commutator of the operators $D$ and $\Phi$.

\begin{definition}\label{D:Calliasopr}
We say that $\DD$ is a \emph{Callias-type operator} if
\begin{enumerate}
\item\label{I:1} $[D,\Phi]$ is a zeroth order differential operator, i.e. a bundle map;
\item\label{I:2} there exist a compact subset $K\Subset M$ and a constant $c>0$ such that
\[
	\Phi^2(x)\;-\;|[D,\Phi](x)|\;\ge\;c
\]
for all $x\in M\setminus K$. Here $|[D,\Phi](x)|$ denotes the operator norm of the linear map $[D,\Phi](x):E_x\to E_x$. In this case, the compact set $K$ is called an \emph{essential support} of $\DD$.
\end{enumerate}
\end{definition}

\begin{remark}\label{R:Calliasequiv}
$\DD$ is a Callias-type operator if and only if $\DD^*$ is.
\end{remark}

\begin{proposition}\label{P:Calliasinv}
Callias-type operators are invertible at infinity in the sense of Definition \ref{D:invatinf}.
\end{proposition}

\begin{proof}
Since $\dm$ is compact, we can always assume that the essential support $K$ contains a neighborhood of $\dm$. Thus for all $u\in C_c^\infty(M,E)$ with ${\rm supp}(u)\cap K=\emptyset$, $u\in C_{cc}^\infty(M,E)$. Then by Proposition \ref{P:greensfor}, \eqref{E:Calliassq}, and Definition \ref{D:Calliasopr},
\[
\begin{aligned}
\|\DD u\|_{L^2(M)}^2&\;=\;(\DD u,\DD u)_{L^2(M)}\;=\;(\DD^*\DD u,u)_{L^2(M)} \\
&\;=\;(D^2u,u)_{L^2(M)}\;+\;((\Phi^2+i[D,\Phi])u,u)_{L^2(M)} \\
&\;\ge\;\|Du\|_{L^2(M)}^2\;+\;c\|u\|_{L^2(M)}^2 \\
&\;\ge\;c\|u\|_{L^2(M)}^2.
\end{aligned}
\]
Therefore $\|\DD u\|_{L^2(M)}\ge\sqrt{c}\|u\|_{L^2(M)}$ and $\DD$ is invertible at infinity.
\end{proof}

\begin{remark}\label{R:Calliasopr}
When $\dm=\emptyset$, $\DD$ has a unique closed extension to $L^2(M,E)$, and it is a Fredholm operator. Thus one can define its $L^2$-index,
\[
\ind\DD\;:=\;\dim\{u\in L^2(M,E):\DD u=0\}-\dim\{u\in L^2(M,E):\DD^*u=0\}.
\]
A seminal result says that this index is equal to the index of a Dirac-type operator (the operator $\partial_+^+$ of \eqref{E:hypersurfaceopr}) on a compact hypersurface outside of the essential support. This was first proved by Callias in \cite{Callias78} for Euclidean space (see also \cite{BottSeeley78}) and was later generalized to manifolds in \cite{Anghel93-2}, \cite{Raade}, \cite{Bunke}, etc. In \cite{BrShi} and \cite{BrCecchini16}, the relationship between such result and cobordism invariance of the index was being discussed for usual and von Neumann algebra cases, respectively.
\end{remark}

\begin{remark}
If $\dm\ne\emptyset$, then in general, $\DD$ is not Fredholm. By Proposition \ref{P:fred}, we need an elliptic boundary condition in order to have a well-defined index and study it.
\end{remark}

\subsection{The APS boundary condition for Callias-type operators}\label{SS:APS4Callias}

We impose the APS boundary condition as discussed in Subsection \ref{SS:APSbc} that enables us to define the index for Callias-type operators.

As in Subsection \ref{SS:APSbc}, we assume the product structure \eqref{E:prostrofD} for $D$ near $\dm$. We also assume that $\Phi$ does not depend on $t$ near $\dm$. Then near $\dm$,
\[
\DD\;=\; c(\tau)\big(\partial_t+A-ic(\tau)\Phi\big)\;=\;c(\tau)\big(\partial_t+\AAA\big),
\]
where $\AAA:=A-ic(\tau)\Phi$ is still formally self-adjoint and thus is an adapted operator to $\DD$. 

Replacing $D$ and $A$ in Subsection \ref{SS:APSbc} by $\DD$ and $\AAA$, we define the APS boundary condition $B_{\rm APS}$ as in \eqref{E:APSbc} for the Callias-type operator $\DD$. It is an elliptic boundary condition. Combining Proposition \ref{P:fred}, Remark \ref{R:Calliasequiv} and Proposition \ref{P:Calliasinv}, we obtain the Fredholmness for the operator $\DD_{B_{\rm APS}}$.

\begin{proposition}\label{P:Calliasfred}
The operator $\DD_{B_{\rm APS}}:\dom\DD_{B_{\rm APS}}\to L^2(M,E)$ is Fredholm, thus has an index
\[
\ind\DD_{B_{\rm APS}}\;=\;\dim\ker\DD_{B_{\rm APS}}-\dim\ker\DD_{B_{\rm APS}^{\rm ad}}^*\;\in\;\ZZ.
\]
\end{proposition}

\subsection{The APS-Callias index theorem}\label{SS:Calliasindthm}

We now formulate the main result of this paper -- a Callias-type index theorem for operators with APS boundary conditions.

By Definition \ref{D:Calliasopr}, the Callias potential $\Phi$ is nonsingular outside of the essential support $K$. Then over $M\setminus K$, there is a bundle decomposition 
\[
	E|_{M\setminus K}\;=\;E_+\oplus E_-,
\] where $E_\pm$ are the positive/negative eigenspaces of $\Phi$. Since Definition \ref{D:Calliasopr}.\ref{I:1} implies that $\Phi$ commutes with Clifford multiplication, $E_\pm$ are also Clifford modules.

Let $L\Subset M$ be a compact subset of $M$ containing the essential support $K$ such that $(K\setminus\dm)\subset\mathring{L}$. Suppose that $\partial L=\dm\sqcup N$, where $N$ is a closed hypersurface partitioning $M$. Denote 
\[
	E_N\;:=\;E|_N,\quad E_{N\pm}\;:=\;E_\pm|_N.
\] 
The restriction of the Clifford multiplication on $E_\pm$ defines a Clifford multiplication $c_N(\cdot)$ on $E_{N\pm}$. Let $\nabla^{E_N}$ be the restriction of the connection $\nabla^E$ on $E$. In general, $\nabla^{E_N}$ does not preserve the decomposition $E_N=E_{N+}\oplus E_{N-}$. However, if we define
\[
\nabla^{E_{N\pm}}\;:=\;{\rm Proj}_{E_{N\pm}}\circ\nabla^{E_N},
\]
where ${\rm Proj}_{E_{N\pm}}$ are the projections onto $T^*N\otimes E_{N\pm}$. One can check that these are Hermitian connections on $E_{N\pm}$ (cf. \cite[Lemma 2.7]{Anghel90}).
Then $E_{N\pm}$ are Clifford modules over $N$, and we define the (formally self-adjoint) Dirac operators on $E_{N\pm}$ by
\[
\partial_\pm\;:=\;\sum_{j=1}^{n-1}c_N(e_j)\nabla_{e_j}^{E_{N\pm}}
\]
at $x\in N$, where $e_1,\dots,e_{n-1}$ is an orthonormal basis of $T_xN$. They can be seen as adapted operators associated to $D_\pm:=D|_{E_\pm}$.

Let $\tau_N$ be a unit inward (with respect to $L$) normal vector field on $N$ and set 
\[
	\nu\;:=\;ic(\tau_N).
\] 
Since $\nu^2=\id$, $\nu$ induces a grading on $E_{N\pm}$
\[
	E_{N\pm}^\pm\;=\;\{u\in E_{N\pm}:\nu u=\pm u\},
\]
It's easy to see from \eqref{E:anticommuting} that $\partial_\pm$ anti-commute with $\nu$. We denote by $\partial_\pm^\pm$ the restrictions of $\partial_\pm$ to $E_{N\pm}^\pm$. Then
\begin{equation}\label{E:hypersurfaceopr}
\partial_\pm^\pm\;:\;C^\infty(N,E_{N\pm}^\pm)\;\to\;C^\infty(N,E_{N\pm}^\mp).
\end{equation}

As mentioned in Remark \ref{R:Calliasopr}, when $\dm=\emptyset$, the classical Callias index theorem asserts that
\[
\ind\DD\;=\;\ind\partial_+^+.
\]
The following theorem generalizes this result to the case of manifolds with boundary.

\begin{theorem}\label{T:Calliasindthm}
Let $\DD=D+i\Phi:C^\infty(M,E)\to C^\infty(M,E)$ be a Callias-type operator on an odd-dimensional complete manifold $M$ with compact boundary $\dm$. Let $B_{\rm APS}$ be the APS boundary condition described in Subsection \ref{SS:APS4Callias}. Then
\begin{equation}\label{E:Calliasindthm}
\ind\DD_{B_{\rm APS}}\;=\;\frac{1}{2}(\ind\partial_+^+-\ind\partial_-^+)-\tilde\eta(\AAA),
\end{equation}
where $\partial_\pm^+:C^\infty(N,E_{N\pm}^+)\to C^\infty(N,E_{N\pm}^-)$ are the Dirac-type operators on the closed manifold $N$,
\begin{equation}\label{E:fulletainv}
\tilde\eta(\AAA)\;:=\;\frac{1}{2}(\dim\ker\AAA+\eta(0;\AAA)),
\end{equation}
and the $\eta$-function $\eta(s;\AAA)$ is defined by
\[
\eta(s;\AAA)\;:=\;\sum_{\lambda\in{\rm spec}(\AAA)\setminus\{0\}}{\rm sign}(\lambda)|\lambda|^{-s}.
\]
\end{theorem}

\begin{remark}\label{R:etainv}
Since $\dm$ is a closed manifold, $\eta(s;\AAA)$ converges absolutely for $\operatorname{Re}(s)$ large. Then $\eta(0;\AAA)$ can be defined using meromorphic continuation of $\eta(s;\AAA)$ and we call it \emph{$\eta$-invariant} for $\AAA$ on $\dm$. Note that $\dm$ is an even-dimensional manifold. In general, the $\eta$-invariant on even-dimensional manifolds is much simpler than on odd-dimensional ones. We refer the reader to \cite{Gilkey85} for details.
\end{remark}

Theorem \ref{T:Calliasindthm} will be proved in the next section. The main idea of the proof is as follows. Recall that we have chosen a compact subset $L$ of $M$ containing the essential support of $\DD$ with boundary $\partial L=\dm\sqcup N$. First use Theorems \ref{T:rit} and \ref{T:splittingthm} to transfer the index we want to find to an index on $L$ with APS boundary condition. Then by APS index formula \cite[Theorem 3.10]{APS} and dimension reason, we get
\[
\ind\DD_{B_{\rm APS}}\;=\;-\tilde\eta(\AAA_N)-\tilde\eta(\AAA).
\]
Then the proof is completed by a careful study of the $\eta$-invariant $\eta(0;\AAA_N)$.

\subsection{Connection between Theorem \ref{T:Calliasindthm} and the usual Callias index theorem}\label{SS:connection}

Consider the special case when $\dm=\emptyset$. Clearly, $\tilde\eta(\AAA)$ vanishes. Since $N=\partial L$ now, by cobordism invariance of the index (see for example \cite[Chapter \upper{\romannumeral17}]{Palais} or \cite{Br-cobinv}),
\[
0\;=\;\ind\partial^+\;=\;\ind\partial_+^+\;+\;\ind\partial_-^+.
\]
Hence $\ind\partial_+^+=-\ind\partial_-^+$, and \eqref{E:Calliasindthm} becomes
\[
\ind\DD\;=\;\ind\partial_+^+,
\]
which is exactly the usual Callias index theorem. Therefore, our Theorem \ref{T:Calliasindthm}  can be seen as a generalization of the Callias index theorem to manifolds with boundary. In particular, we give a new proof of the Callias index theorem for manifolds without boundary.

\subsection{An asymmetry result}\label{SS:asymetry}

{One can see from \eqref{E:APSbc} and \eqref{E:adAPS} that the APS boundary condition $B_{\rm APS}$ involves spectral projection onto $(-\infty,0)$, while its adjoint boundary condition $B_{\rm APS}^{\rm ad}$ involves spectral projection onto a slightly different interval $(-\infty,0]$. This shows that Atiyah-Patodi-Singer boundary condition is not symmetric.
When the manifold $M$ is compact, this asymmetry can be expressed in terms of the kernel of the adapted operator (cf. \cite[pp. 58-60]{APS}). For our Callias-type operator on noncompact manifold, a similar result still holds. To avoid confusion of notations, we use $D+i\Phi$ for $\DD$ and $D-i\Phi$ for $\DD^*$.

\begin{corollary}\label{C:asymmetry}
Under the same hypothesis as in Theorem \ref{T:Calliasindthm},
\[
\ind(D+i\Phi)_{B_{\rm APS}}+\ind(D-i\Phi)_{B_{\rm APS}}\;=\;-\dim\ker\AAA.
\]
\end{corollary}

\begin{proof}
Recall that $D+i\Phi$ and $D-i\Phi$ can be written as
\[
\begin{aligned}
D+i\Phi&\;=\;c(\tau)\big(\partial_t+\AAA\big), \\
D-i\Phi&\;=\;c(\tau)\big(\partial_t+\tilde\AAA\big),
\end{aligned}
\]
where the adapted operators $\AAA$ and $\tilde\AAA$ satisfy
\begin{equation}\label{E:anticommuting1}
\tilde\AAA\circ c(\tau)\;=\;-c(\tau)\circ\AAA.
\end{equation}

Apply Theorem \ref{T:Calliasindthm} to $D+i\Phi$ and $D-i\Phi$. Notice that $\partial_+^+$ and $\partial_-^+$ are interchanged for these two Callias-type operators, so we have
\[
\begin{aligned}
\ind(D+i\Phi)_{B_{\rm APS}}&\;=\;\frac{1}{2}(\ind\partial_+^+-\ind\partial_-^+)-\tilde\eta(\AAA), \\
\ind(D-i\Phi)_{B_{\rm APS}}&\;=\;\frac{1}{2}(\ind\partial_-^+-\ind\partial_+^+)-\tilde\eta(\tilde\AAA).
\end{aligned}
\]
Add them up and it suffices to show that
\[
\tilde\eta(\AAA)\,+\,\tilde\eta(\tilde\AAA)\;=\;\dim\ker\AAA.
\]

By \eqref{E:anticommuting1}, the map $c(\tau)$ sends eigensections of $\AAA$ associated with eigenvalue $\lambda_j$ to eigensections of $\tilde\AAA$ associated with eigenvalue $-\lambda_j$ bijectively and vice versa. In particular, it induces an isomorphism between the kernel of $\AAA$ and that of $\tilde\AAA$. So
\[
\eta(0;\AAA)+\eta(0;\tilde\AAA)=0\qquad\mbox{and}\qquad\dim\ker\AAA+\dim\ker\tilde\AAA=2\dim\ker\AAA.
\]
Now the corollary follows from \eqref{E:fulletainv}.
\end{proof}

Notice, that if $\dm=\emptyset$ then Corollary~\ref{C:asymmetry} implies a well known result (cf. for example, \cite[(2.10)]{BrCecchini16})  
\[
	\ind(D+i\Phi) \;= \;-\ind(D-i\Phi).
\]

\section{Proof of Theorem \ref{T:Calliasindthm}}\label{S:proof}

We prove Theorem \ref{T:Calliasindthm} following the idea sketched in Subsection \ref{SS:Calliasindthm}. To simplify notations, we will write $\ind\DD$ for $\ind\DD_{B_{\rm APS}}$ in this section.

\subsection{Deformation of structures near $N$}\label{SS:deformation}

Remember we assumed that $(K\setminus\dm)\subset\mathring{L}$, so there exists a relatively compact neighborhood $U(N)$ of $N$ which does not intersect with the essential support $K$. Out first step is to do deformation on $U(N)$. The following lemma is from \cite[Section 6]{BrCecchini16}.

\begin{lemma}\label{L:deformation}
Set $N_\delta:=N\times(-\delta,\delta)$ for any $\delta>0$. One can deform all the structures in the neighborhood $U(N)$ of $N$ so that the following conditions are satisfied:
\begin{enumerate}
\item\label{I:1} $U(N)$ is isometric to $N_{2\varepsilon}$;
\item\label{I:2} {\rm (see also \cite[Lemma 5.3]{BrCecchini16})} the restrictions of the Clifford modules $E|_{N_\varepsilon}$ and $E_\pm|_{N_\varepsilon}$ are isomorphic to the pull backs of $E_N$ and $E_{N\pm}$ to $N_\varepsilon$ respectively along with connections;
\item\label{I:3} {\rm (see also \cite[Lemma 5.4]{BrCecchini16})} $\Phi|_{N_\varepsilon}$ is a constant multiple of its unitarization $\Phi_0:=\Phi(\Phi^2)^{-1/2}$, i.e., $\Phi|_{E_\pm}=\pm h$ on $N_\varepsilon$, where $h>0$ is a constant;
\item\label{I:4} the potential $V$  from \eqref{E:Diracopr} of the Dirac-type operator $D$ vanishes on $N_\varepsilon$;
\item\label{I:5} $\DD$ is always a Callias-type operator throughout the deformation, and the essential support of the Callias-type operator associated to the new structures is still contained in $L\setminus(N\times(-\varepsilon,0])$.
\end{enumerate}
\end{lemma}

\begin{remark}\label{R:deformation1}
As a result of \ref{I:1} and \ref{I:2}, we can write
\[
\begin{aligned}
D|_{N_\varepsilon}&\;=\;c(\tau_N)(\partial_t+\partial), \\
D_\pm|_{N_\varepsilon}&\;=\;c(\tau_N)(\partial_t+\partial_\pm),
\end{aligned}
\]
where $t$ is the normal coordinate pointing inward $L$ and $\partial,\partial_\pm$ are as in Subsection \ref{SS:Calliasindthm}. Furthermore, by \ref{I:3}, our Callias-type operator has the form
\begin{equation}\label{E:product}
\DD|_{N_\varepsilon}\;=\;c(\tau_N)(\partial_t+\AAA_N),
\end{equation}
where $\AAA_N=\partial-ic(\tau_N)\Phi|_N$ does not depend on $t$ and 
\begin{equation}\label{E:adaponN}
\AAA_N|_{E_{N\pm}}\;=\;\partial_\pm\mp ic(\tau_N)h.
\end{equation}
It's also easy to see from \ref{I:3} that
\begin{equation}\label{E:defCalliaspotential}
\Phi^2\;=\;h^2\quad\mbox{and}\quad[D,\Phi]|_{E_\pm}\;=\;[D,\pm h]\;=\;0
\end{equation}
on $N_\varepsilon$.
\end{remark}

\begin{remark}\label{R:deformation2}
Below in Lemma \ref{L:indonL} we use the freedom to choose $h$ in \ref{I:3} to be arbitrarily large.
\end{remark}

\begin{proposition}\label{P:preservingind}
The deformation in Lemma \ref{L:deformation} preserves the index of the Callias-type operator $\DD=D+i\Phi$ under APS boundary condition.
\end{proposition}

\begin{proof}
Let $W$ be the closure of $U(N)$. It is a compact subset of $M$ which does not intersect the boundary $\dm$. This indicates that $\DD$ keeps unchanged near the boundary and one can impose the same APS boundary condition. Since the deformation only occurs on the compact set $W$ and is continuous, the domain of $\DD$ remains the same under APS boundary condition. Therefore, throughout the deformation, $\DD$ is always a bounded operator from this fixed domain to $L^2(M,E)$ which is Fredholm. Now by the stability of the Fredholm index (cf. \cite[Proposition \upper{\romannumeral3}.7.1]{LM}), the index of $\DD$ is preserved.
\end{proof}

\begin{remark}\label{R:preservingind}
One can also show Proposition \ref{P:preservingind} by using relative index theorem and the fact that the dimension is odd.
\end{remark}

Proposition \ref{P:preservingind} ensures that we can make the following assumption.

\begin{assumption}\label{A:deformation}
We assume that conditions \ref{I:1}-\ref{I:5} of Lemma \ref{L:deformation} are satisfied for our problem henceforth.
\end{assumption}

\subsection{The index on manifold with a cylindrical end}\label{SS:cylinder}

Note that $N$ gives a partition $M=L\cup_N(M\setminus L)$. Consider $M_1=N\times (-\infty,\infty)$ with the partition
\[
	M_1 \;=\;(N\times(-\infty,0])\cup_N(N\times(0,\infty)).
\] 
Lift the Clifford module $E_N$, Dirac operator $\partial$ and restriction of bundle map $\Phi|_N$ from $N$ to $M_1$. By Assumption \ref{A:deformation}, there are isomorphisms between structures of $M$ near $N$ and those of $M_1$ near $N\times\{0\}$. One can do the ``cut-and-glue'' procedure as described in Subsection \ref{SS:rit} to form
\begin{equation}\label{E:cutandglue}
\hat{M}\;=\;L\cup_N(N\times(0,\infty)),\qquad M_2\;=\;(N\times(-\infty,0])\cup_N(M\setminus L).
\end{equation}
We obtain Callias-type operators $\DD,\DD_1,\hDD,\DD_2$ acting on $E,E_1,\hat{E},E_2$ over corresponding manifold. They satisfy Theorem \ref{T:rit}. Therefore
\[
	\ind\DD+\ind\DD_1\;=\;\ind\hDD+\ind\DD_2.
\]
Notice that $\DD_1$ and $\DD_2$ are Callias-type operators with empty essential supports on manifolds without boundary. Therefore, $\DD_1$, $\DD_2$ and their adjoints are invertible operators. So $\ind\DD_1=\ind\DD_2=0$ and we get
\begin{lemma}\label{L:indD=indDhat}
$\ind\DD\;=\;\ind\hDD$.
\end{lemma}

\begin{remark}\label{R:cylindricalend}
Now the problem is moved to $\hat{M}$, a manifold with a cylindrical end. We point out that conditions \ref{I:2}-\ref{I:4} of Lemma \ref{L:deformation} continue holding on the cylindrical end.
\end{remark}

\subsection{Applying the splitting theorem}\label{SS:appsplittingthm}

We have the partition of $\hat{M}$ as in \eqref{E:cutandglue} and $\hDD$ is of form \eqref{E:product} near $N$. Cut $\hat{M}$ along $N$. Define boundary condition on $L$ along $N$ to be the APS boundary condition $H_{(-\infty,0)}^{1/2}(\AAA_N)$, and the boundary condition on $N\times[0,\infty)$ along $N$ to be $H_{[0,\infty)}^{1/2}(\AAA_N)$. Denote by $\hDD_1$ and $\hDD_2$ the restrictions of $\hDD$ to $L$ and $N\times[0,\infty)$, respectively. Let $\ind\hDD_1$ be the index of $\hDD_1$ with APS boundary condition and $\ind\hDD_2$ be the index of $\hDD_2$ with boundary condition $H_{[0,\infty)}^{1/2}(\AAA_N)$. Then by Theorem \ref{T:splittingthm},
\[
\ind\hDD\;=\;\ind\hDD_1+\ind\hDD_2.
\]

\begin{lemma}\label{L:indonL}
$\ind\hDD\;=\;\ind\hDD_1$.
\end{lemma}

\begin{proof}
We need to prove that $\ind\hDD_2=0$. Remember that $\hDD_2=\hat{D}+i\hat\Phi$ satisfies conditions of Lemma \ref{L:deformation} on $N\times[0,\infty)$. For any $u\in C_c^\infty(N\times[0,\infty),\hat{E})$ satisfying $u|_N\in H_{[0,\infty)}^{1/2}(\AAA_N)$, by Proposition \ref{P:greensfor}, \eqref{E:product}, \eqref{E:defCalliaspotential} and Remark \ref{R:cylindricalend},
\[
\begin{aligned}
\|\hDD_2u\|_{L^2}^2&\;=\;(\hDD_2u,\hDD_2u)_{L^2} \\
&\;=\;(\hDD_2^*\hDD_2u,u)_{L^2}-\int_N\langle c(\tau_N)\hDD_2u,u\rangle dS \\
&\;=\;(\hat{D}^2u,u)_{L^2}+((\hat\Phi^2+i[\hat{D},\hat\Phi])u,u)_{L^2}+\int_N(\langle \partial_tu,u\rangle+\langle\AAA_Nu,u\rangle)dS \\
&\;\ge\;(\hat{D}^2u,u)_{L^2}+h^2\|u\|_{L^2}^2+\int_N\langle \partial_tu,u\rangle dS.
\end{aligned}
\]

By Assumption \ref{A:deformation} and Remark \ref{R:cylindricalend}, the potential $\hat V$ for $\hat D$ vanishes on $N\times[0,\infty)$. The Weitzenb\"ock identity (or general Bochner identity, cf. \cite[Proposition \upper{\romannumeral2}.8.2]{LM}) for Dirac operator gives that
\[
\hat{D}^2\;=\;\hat\nabla^*\hat\nabla+\hat{\mathcal{R}}\quad\mbox{on }N\times[0,\infty),
\]
where the bundle map term $\hat{\mathcal{R}}$ is the curvature transformation associated with the Clifford module $\hat{E}|_{N\times[0,\infty)}$. Since this bundle is the lift of $E_N$ from the compact base $N$, $\hat{\mathcal{R}}$ is bounded on $N\times[0,\infty)$. As mentioned in Remark \ref{R:deformation2}, one can choose $h$ large enough so that $h^2/2$ is greater than the upper bound of the norm $|\hat{\mathcal{R}}|$. Applying Proposition \ref{P:greensfor} to $\hat\nabla$, we have
\begin{multline*}
(\hat\nabla^*\hat\nabla u,u)_{L^2}-\|\hat\nabla u\|_{L^2}^2\;=\;\int_N\langle\sigma_{\hat\nabla^*}(\tau_N)\hat\nabla u,u\rangle dS \\
\;=\;-\int_N\langle\hat\nabla u,\sigma_{\hat\nabla}(\tau_N)u\rangle dS 
\;=\;-\int_N\langle\hat\nabla u,\tau_N\otimes u\rangle dS \\
\;=\;-\int_N\langle\hat\nabla_{\tau_N}u,u\rangle dS 
\;=\;-\int_N\langle \partial_tu,u\rangle dS.
\end{multline*}
Thus
\begin{multline*}
\|\hDD_2u\|_{L^2}^2\;\ge\;(\hat\nabla^*\hat\nabla u,u)_{L^2}+
\int_N\langle \partial_tu,u\rangle dS+h^2\|u\|_{L^2}^2+(\hat{\mathcal{R}}u,u)_{L^2} \\
\;\ge\;\|\hat\nabla u\|_{L^2}^2+\frac{h^2}{2}\|u\|_{L^2}^2 
\;\ge\;\frac{h^2}{2}\|u\|_{L^2}^2.
\end{multline*}
Therefore, $\hDD_2$ is invertible on the domain determined by the boundary condition $H_{[0,\infty)}^{1/2}(\AAA_N)$ and $\ker\hDD_2=\{0\}$. Similarly, $\ker(\hDD_2)^{\rm ad}=\{0\}$. Hence $\ind\hDD_2=0$ and $\ind\hDD=\ind\hDD_1$.
\end{proof}

Standard Atiyah-Patodi-Singer index formula (\cite[Theorem 3.10]{APS}) applies to $\ind\hDD_1$ giving that
\begin{equation}\label{E:APSindthm}
\ind\hDD_1\;=\;\int_LAS-\tilde\eta(\AAA_N)-\tilde\eta(\AAA),
\end{equation}
where $AS$ is the interior Atiyah-Singer integrand. Since the dimension of $L$ is odd, this integral vanishes. Combining Lemmas \ref{L:indD=indDhat}, \ref{L:indonL} and \eqref{E:APSindthm}, we finally obtain
\begin{equation}\label{E:ind=eta}
\ind\DD\;=\;-\tilde\eta(\AAA_N)-\tilde\eta(\AAA).
\end{equation}

\subsection{The $\eta$-invariant of the perturbed Dirac operator on $N$}\label{SS:etainv}

In the last subsection, we have expressed the index of $\DD_{B_{\rm APS}}$ in terms of $\tilde\eta(\AAA)$ and $\tilde\eta(\AAA_N)$ as in \eqref{E:ind=eta}, where
\[
\AAA_N\;=\;\AAA_{N+}\oplus\AAA_{N-}\;=\;(\partial_+-\nu h)\oplus(\partial_-+\nu h)
\]
under the splitting $E_N=E_{N+}\oplus E_{N-}$, and $\nu=ic(\tau_N)$ (cf. \eqref{E:adaponN}). In this subsection, we shall show how $\tilde\eta(\AAA_N)$ can be written as the difference of two indexes as in the right-hand side of \eqref{E:Calliasindthm}.

Recall that
\begin{equation}\label{E:tildeetainv}
\tilde\eta(\AAA_N)\;=\;\frac{1}{2}(\dim\ker\AAA_N+\eta(0;\AAA_N)).
\end{equation}
$\AAA_N$ and $\partial$  can be viewed as adapted operators to $\DD$ and $D$ on $N$, respectively. Using the fact that $\partial$ anti-commutes with $\nu$, we have
\begin{equation}\label{E:A_N^2}
\AAA_N^2\;=\;\partial^2+h^2.
\end{equation}
Since $h>0$ is a constant, $\AAA_N$ is an invertible operator, and hence
\begin{equation}\label{E:dimker=0}
\dim\ker\AAA_N\;=\;0.
\end{equation}
As for $\eta(0;\AAA_N)$, we have the following lemma.

\begin{lemma}\label{L:etainv}
$\eta(0;\AAA_N)=-\ind\partial_+^++\ind\partial_-^+$.
\end{lemma}

\begin{proof}
Notice that $\AAA_N$ is a perturbation of $\partial$ by a bundle map $\nu$ which anti-commutes with it. Restricting to $E_{N+}$, we write $\AAA_N$ according to the grading $E_{N+}=E_{N+}^+\oplus E_{N+}^-$ induced by $\nu$ (see Subsection \ref{SS:Calliasindthm}),
\[
\AAA_{N+}\;=\;
\begin{bmatrix}
-h & \partial_+^- \\
\partial_+^+ & h
\end{bmatrix}.
\]

The spectrum of $\AAA_{N+}$  consists of eigenvalues with finite multiplicity. By \eqref{E:A_N^2}, the eigenvalues of $\AAA_N$ have absolute value of at least $h$. Suppose that $u=u^+\oplus u^-\in C^\infty(N,E_{N+})$ is an eigenvector of $\AAA_{N+}$ with eigenvalue $\lambda$. Then
\[
\lambda
\begin{bmatrix}
u^+ \\
u^-
\end{bmatrix}
\;=\;\AAA_{N+}
\begin{bmatrix}
u^+ \\
u^-
\end{bmatrix}\;=\;
\begin{bmatrix}
-h & \partial_+^- \\
\partial_+^+ & h
\end{bmatrix}
\begin{bmatrix}
u^+ \\
u^-
\end{bmatrix}\;=\;
\begin{bmatrix}
\partial_+^-u^--hu^+ \\
\partial_+^+u^++hu^-
\end{bmatrix},
\]
which gives
\begin{equation}\label{E:speceq}
\left\{
\begin{array}{l}
\partial_+^-u^-\;=\;(\lambda+h)u^+ \\
\partial_+^+u^+\;=\;(\lambda-h)u^-
\end{array}
\right..
\end{equation}
Then
\[
\AAA_{N+}
\begin{bmatrix}
(\lambda+h)u^+ \\
-(\lambda-h)u^-
\end{bmatrix}\;=\;
\begin{bmatrix}
-(\lambda-h)\partial_+^-u^--h(\lambda+h)u^+ \\
(\lambda+h)\partial_+^+u^+-h(\lambda-h)u^-
\end{bmatrix}
\;=\;-\lambda
\begin{bmatrix}
(\lambda+h)u^+ \\
-(\lambda-h)u^-
\end{bmatrix}.
\]
Note that the map $u^+\oplus u^-\mapsto(\lambda+h)u^+\oplus(-(\lambda-h)u^-)$ is invertible when $|\lambda|>h$. Therefore, for such $\lambda$, this map induces an isomorphism between the eigenspaces of $\AAA_{N+}$ corresponding to eigenvalues $\lambda$ and $-\lambda$. This means that the spectrum of $\AAA_{N+}$ lying in $(-\infty,-h)$ is symmetric to that lying in $(h,\infty)$, hence
\[
\eta(0;\AAA_{N+})\;=\;\dim\ker(\AAA_{N+}-h)\;-\;\dim\ker(\AAA_{N+}+h).
\]

If $u^+\oplus u^-\in\ker(\AAA_{N+}-h)$, by letting $\lambda=h$ in \eqref{E:speceq}, we get
\[
\left\{
\begin{array}{l}
\partial_+^-u^-\;=\;2hu^+ \\
\partial_+^+u^+\;=\;0
\end{array}
\right..
\]
{Applying $\partial_+^+$ to the first equation yields $(\partial_+^+\partial_+^-)u^-=0$}.
{Thus $u^-\in\ker(\partial_+^+\partial_+^-)$. Since $\partial_+$ is formally self-adjoint, $\ker(\partial_+^+\partial_+^-)=\ker\partial_+^-$.} So $u^-\in\ker\partial_+^-$ and $u^+=0$. Therefore
\[
\ker(\AAA_{N+}-h)\;=\;\{0\oplus u^-:u^-\in\ker\partial_+^-\}.
\]
Hence $\dim\ker(\AAA_{N+}-h)=\dim\ker\partial_+^-$. Similarly, $\dim\ker(\AAA_{N+}+h)=\dim\ker\partial_+^+$. Then
\[
\eta(0;\AAA_{N+})\;=\;\dim\ker\partial_+^-\;-\;\dim\ker\partial_+^+\;=\;-\ind\partial_+^+.
\]

The discussion on $E_{N-}$ is exactly the same as what we just did on $E_{N+}$. One gets
\[
\eta(0;\AAA_{N-})\;=\;\ind\partial_-^+.
\]
As a direct sum of $\AAA_{N+}$ and $\AAA_{N-}$, by the additivity of the $\eta$-invariant, finally we obtain
\[
\eta(0;\AAA_N)\;=\;-\ind\partial_+^++\ind\partial_-^+.
\]
\end{proof}

Now \eqref{E:Calliasindthm} follows simply from \eqref{E:ind=eta}, \eqref{E:tildeetainv}, \eqref{E:dimker=0} and Lemma \ref{L:etainv}. We complete the proof of Theorem \ref{T:Calliasindthm}.



\begin{thebibliography}{99}

\bibitem{Anghel90}
N.~Anghel.
\newblock {$L^2$}-index formulae for perturbed {D}irac operators.
\newblock {\em Comm. Math. Phys.}, 128(1):77--97, 1990.

\bibitem{Anghel93}
N.~Anghel.
\newblock An abstract index theorem on noncompact {R}iemannian manifolds.
\newblock {\em Houston J. Math.}, 19(2):223--237, 1993.

\bibitem{Anghel93-2}
N.~Anghel.
\newblock On the index of {C}allias-type operators.
\newblock {\em Geom. Funct. Anal.}, 3(5):431--438, 1993.

\bibitem{AB}
M.F. Atiyah and R~Bott.
\newblock The index problem for manifolds with boundary.
\newblock In {\em Differential {A}nalysis, {B}ombay {C}olloq., 1964}, pages
  175--186. Oxford Univ. Press, London, 1964.

\bibitem{APS}
M.F. Atiyah, V.K. Patodi, and I.M. Singer.
\newblock Spectral asymmetry and {R}iemannian geometry. {I}.
\newblock {\em Math. Proc. Cambridge Philos. Soc.}, 77:43--69, 1975.

\bibitem{BB-1}
C.~B{\"a}r and W.~Ballmann.
\newblock Boundary value problems for elliptic differential operators of first
  order.
\newblock In {\em Surveys in differential geometry. {V}ol. {XVII}}, volume~17
  of {\em Surv. Differ. Geom.}, pages 1--78. Int. Press, Boston, MA, 2012.

\bibitem{BB-2}
C.~B{\"a}r and W.~Ballmann.
\newblock Guide to elliptic boundary value problems for {D}irac-type operators.
\newblock In {\em Arbeitstagung {B}onn 2013}, volume 319 of {\em Progr. Math.},
  pages 43--80. Birkh\"auser/Springer, Cham, 2016.

\bibitem{BW}
B.~Boo{\ss}-Bavnbek and K.P. Wojciechowski.
\newblock {\em Elliptic boundary problems for {D}irac operators}.
\newblock Mathematics: Theory \& Applications. Birkh\"auser Boston, Inc.,
  Boston, MA, 1993.

\bibitem{BottSeeley78}
R.~Bott and R.~Seeley.
\newblock Some remarks on the paper of {C}allias: ``{A}xial anomalies and index
  theorems on open spaces'' [{C}omm. {M}ath. {P}hys. {\bf 62} (1978), no. 3,
  213--234;\ {MR} 80h:58045a].
\newblock {\em Comm. Math. Phys.}, 62(3):235--245, 1978.

\bibitem{Br-cobinv}
M.~Braverman.
\newblock New proof of the cobordism invariance of the index.
\newblock {\em Proc. Amer. Math. Soc.}, 130(4):1095--1101, 2002.

\bibitem{BrCecchini16}
M.~Braverman and S.~Cecchini.
\newblock Callias-type operators in von Neumann algebras.
\newblock {\em The Journal of Geometric Analysis}, pages 1--41, 2017.

\bibitem{BrShi}
M.~Braverman and P.~Shi.
\newblock Cobordism invariance of the index of {C}allias-type operators.
\newblock {\em Comm. Partial Differential Equations}, 41(8):1183--1203, 2016.

\bibitem{BrShi17}
M.~Braverman and P.~Shi.
\newblock The {A}tiyah-{P}atodi-{S}inger index on manifolds with non-compact boundary.
\eprint {arXiv:1706.06737 [math.DG]}, 2017.

\bibitem{BrShi17-2}
M.~Braverman and P.~Shi.
\newblock APS index theorem for even-dimensional manifolds with non-compact boundary.
\eprint {arXiv:1708.08336 [math.DG]}, 2017.

\bibitem{BruningMoscovici}
J.~Bruning and H.~Moscovici.
\newblock {$L^2$}-index for certain {D}irac-{S}chr\"odinger operators.
\newblock {\em Duke Math. J.}, 66(2):311--336, 1992.

\bibitem{Bunke}
U.~Bunke.
\newblock A {$K$}-theoretic relative index theorem and {C}allias-type {D}irac
  operators.
\newblock {\em Math. Ann.}, 303(2):241--279, 1995.

\bibitem{Callias78}
C.~Callias.
\newblock Axial anomalies and index theorems on open spaces.
\newblock {\em Comm. Math. Phys.}, 62(3):213--234, 1978.

\bibitem{CarvalhoNistor14}
C.~Carvalho and V.~Nistor.
\newblock An index formula for perturbed {D}irac operators on {L}ie manifolds.
\newblock {\em J. Geom. Anal.}, 24(4):1808--1843, 2014.

\bibitem{Chernoff}
P.R. Chernoff.
\newblock Essential self-adjointness of powers of generators of hyperbolic
  equations.
\newblock {\em J. Functional Analysis}, 12:401--414, 1973.

\bibitem{Gilkey85}
P.B. Gilkey.
\newblock The eta invariant for even-dimensional {${\rm PIN}_{{\rm c}}$}
  manifolds.
\newblock {\em Adv. in Math.}, 58(3):243--284, 1985.

\bibitem{GL}
M.~Gromov and H.B. Lawson.
\newblock Positive scalar curvature and the {D}irac operator on complete
  {R}iemannian manifolds.
\newblock {\em Inst. Hautes \'Etudes Sci. Publ. Math.}, (58):83--196, 1983.

\bibitem{GN}
N.~Grosse and R.~Nakad.
\newblock Boundary value problems for noncompact boundaries of {${\rm
  Spin}^{\rm c}$} manifolds and spectral estimates.
\newblock {\em Proc. Lond. Math. Soc. (3)}, 109(4):946--974, 2014.

\bibitem{Kottke11}
C.~Kottke.
\newblock An index theorem of {C}allias type for pseudodifferential operators.
\newblock {\em J. K-Theory}, 8(3):387--417, 2011.

\bibitem{Kottke15}
C.~Kottke.
\newblock A {C}allias-type index theorem with degenerate potentials.
\newblock {\em Comm. Partial Differential Equations}, 40(2):219--264, 2015.

\bibitem{LM}
H.B. Lawson and M.L. Michelsohn.
\newblock {\em Spin geometry}, volume~38 of {\em Princeton Mathematical
  Series}.
\newblock Princeton University Press, Princeton, NJ, 1989.

\bibitem{Palais}
R.S. Palais.
\newblock {\em Seminar on the {A}tiyah-{S}inger index theorem}.
\newblock With contributions by M. F. Atiyah, A. Borel, E. E. Floyd, R. T.
  Seeley, W. Shih and R. Solovay. Annals of Mathematics Studies, No. 57.
  Princeton University Press, Princeton, N.J., 1965.

\bibitem{Raade}
J.~R{\aa}de.
\newblock Callias' index theorem, elliptic boundary conditions, and cutting and
  gluing.
\newblock {\em Comm. Math. Phys.}, 161(1):51--61, 1994.

\bibitem{Wimmer14}
R.~Wimmer.
\newblock An index for confined monopoles.
\newblock {\em Comm. Math. Phys.}, 327(1):117--149, 2014.

\end{thebibliography}
\end{document}